\documentclass[reqno,english]{amsart}
\usepackage{amsfonts,amsmath,latexsym,verbatim,amscd,mathrsfs,color,array}
\usepackage[colorlinks=true]{hyperref}
\usepackage{amsmath,amssymb,amsthm,amsfonts,graphicx,color}
\usepackage{amssymb}
\usepackage{pdfsync}
\usepackage{epstopdf}
\usepackage{cite}
\usepackage{graphicx}
\usepackage{datetime}

\usepackage{caption2}
\usepackage{subfigure}
\usepackage{wrapfig}

\newtheorem{definition}{Definition}[section]
\newtheorem{theorem}{Theorem}[section]
\newtheorem{lemma}{Lemma}[section]

\newtheorem{prop}{Proposition}[section]
\newtheorem{remark}{Remark}[section]

\numberwithin{equation}{section}

\begin{document}
\title[Optimal $H_{\infty}$ control]{Optimal $H_{\infty}$ control based on stable manifold of discounted Hamilton-Jacobi-Isaacs equation}
\author[G. Chen]{Guoyuan Chen}
\address{\noindent
School of Data Sciences, Zhejiang University of Finance \& Economics, Hangzhou 310018, Zhejiang, P. R. China}
\email{gychen@zufe.edu.cn}

\author[Y. Wang]{Yi Wang}
\address{\noindent
School of Data Sciences, Zhejiang University of Finance \& Economics, Hangzhou 310018, Zhejiang, P. R. China}
\email{wangyihzh@gmail.com}

\author[Q. Zhou]{Qinglong Zhou}
\address{\noindent
School of Mathematical Science, Zhejiang University, Hangzhou 310058, Zhejiang, P. R. China}
\email{zhouqinglong@zju.edu.cn}

\begin{abstract}
The optimal \(H_{\infty}\) control problem over an infinite time horizon, which incorporates a performance function with a discount factor \(e^{-\alpha t}\) (\(\alpha > 0\)), is important in various fields. Solving this optimal \(H_{\infty}\) control problem is equivalent to addressing a discounted Hamilton-Jacobi-Isaacs (HJI) partial differential equation. In this paper, we first provide a precise estimate for the discount factor \(\alpha\) that ensures the existence of a nonnegative stabilizing solution to the HJI equation. This stabilizing solution corresponds to the stable manifold of the characteristic system of the HJI equation, which is a contact Hamiltonian system due to the presence of the discount factor. Secondly, we demonstrate that approximating the optimal controller in a natural manner results in a closed-loop system with a finite \(L_2\)-gain that is nearly less than the gain of the original system. Thirdly, based on the theoretical results obtained, we propose a deep learning algorithm to approximate the optimal controller using the stable manifold of the contact Hamiltonian system associated with the HJI equation. Finally, we apply our method to the \(H_{\infty}\) control of the Allen-Cahn equation to illustrate its effectiveness.
\end{abstract}
\maketitle

\begin{flushright}
\xxivtime \, \today
\end{flushright}

\section{Introduction}

\subsection{Background}
Robust feedback control is a control system design approach aimed at ensuring the stability and performance of a system amid uncertainties and variations. It involves the development of control functions capable of addressing changes in system dynamics, disturbances, parameter variations, and model uncertainties. In this paper, we focus on \(H_{\infty}\) control, which has been widely utilized to mitigate the impact of disturbances on system performance, as discussed in sources such as \cite{zames1981feedback, doyle1989state, van19922, van2000l2, bacsar2008h}. Our primary concern is the optimal $H_{\infty}$ control problem over a continuous infinite time horizon, with a performance function that incorporates a discount factor \(e^{-\alpha t}\) for some \(\alpha> 0\). This discount factor indicates that future costs are deemed less significant than equivalent costs incurred at the present time. Such discounted performance functions are commonly applied in various domains, including optimal control \cite{postoyan2016stability}, dynamic programming and reinforcement learning \cite{bertsekas2019reinforcement,gaitsgory2015stabilization} as well as in robust control problems \cite{modares2015h}. In more complex scenarios, like infinite-dimensional control problems over an infinite time horizon, the inclusion of a discount factor in the performance function may naturally result in a differentiable control-to-state mapping in appropriately chosen spaces, as analyzed recently by Casas and Kunisch \cite{casas2023infinite, casas2023infinite2}.

It is well known that finding the optimal input for $H_{\infty}$ control for continuous-time systems is equivalent to solving an associated Hamilton-Jacobi-Isaacs (HJI) equation, which is a specific type of Hamilton-Jacobi (HJ) equation. In general, HJ equations do not have analytical solutions, making the search for numerical solutions critically important.
A significant amount of research has been devoted to numerically solving various HJ equations (see, for example, \cite{cacace2012patchy, darbon2016algorithms, kalise2018polynomial, jiang2016using, beard1997galerkin, mceneaney2007curse, ohtsuka2010solutions, kalise2020robust}). However, as noted by \cite{nakamura2019adaptive}, these existing methods may face several limitations, including high computational costs for high-dimensional systems, difficulties in estimating solution accuracy for general systems, solutions that are restricted to a small neighborhood around a fixed point or nominal trajectory, and the requirement for specific structures in the nonlinear terms of the system.

Recently, various deep learning methods have been developed to efficiently solve HJ equations in high-dimensional domains using grid-free sampling. For instance, \cite{sirignano2018dgm} utilizes neural networks (NNs) to approximate HJ equation solutions by minimizing PDE and boundary condition residuals at randomly sampled collocation points. \cite{han2018solving} and \cite{raissi2018forward} reformulate PDEs, including HJB equations, as stochastic differential equations for deep learning solutions. Additionally, \cite{kunisch2021semiglobal} provides a theoretical framework for using NNs to approximate optimal feedback laws in functional spaces. Furthermore, \cite{nakamura2019adaptive} and \cite{kang2019algorithms} introduce a causality-free, data-driven deep learning algorithm for solving HJ equations, enhancing NN training through adaptive data generation. For a comprehensive review, see \cite{E2021algorithms}.

Unlike the direct solution of HJ equations as mentioned earlier, methods developed by \cite{van1991state} and \cite{sakamoto2008} address infinite horizon optimal control and \(H_\infty\) control problems by finding stabilizing solutions to stationary HJ equations that correspond to the stable manifolds of the characteristic Hamiltonian systems at specific equilibria. Once the stable manifold is obtained, the optimal control can be directly represented by the stable manifold (\cite{sakamoto2008}), bypassing the need for the gradient of the HJB equation solution. This approach can be viewed as a natural extension of LQ theory to nonlinear systems.

Computing stable manifolds for HJ equations is generally challenging. \cite{sakamoto2008} proposed an iterative procedure to approximate the characteristic Hamiltonian system's trajectories on the stable manifold, yielding promising feedback controls for certain initial conditions or near nominal trajectories (\cite{sakamoto2013case,horibe2017optimal, horibe2018nonlinear}). However, this method can be time-consuming for more general initial conditions or high-dimensional systems.
To address these issues, \cite{chen2020symplectic} introduced a symplectic algorithm to enlarge the effective domain. Despite the method's success in specific applications, theoretical results on the conditions under which the approximate feedback control is effective are limited. \cite{chen2020deep} provided a proof that under natural conditions on approximation errors, the derived controller is nearly optimal and stabilizes the system. Additionally, \cite{chen2020deep} introduced a grid-free deep learning algorithm capable of handling high-dimensional problems to find such an approximate stable manifold inspired by \cite{nakamura2019adaptive}.

\subsection{Main contributions}
In this paper, we aim to identify the optimal $H_{\infty}$ control for the robust control problem over an infinite horizon with a discounted performance function. This challenge corresponds to solving a discounted HJI equation. As demonstrated in \cite{sakamoto2008}, the optimal input can be derived from the stable manifold within a neighborhood of the equilibrium. However, such kind of optimal $H_{\infty}$ control problem presents individual difficulties compared to just optimal control. For instance, what is the appropriate range for the discount factor to ensure the existence of a stable manifold that can generate the optimal $H_{\infty}$ control? What conditions should be satisfied to ensure that the approximation of the stable manifold retains finiteness of the $L_2$-gain as the exact one? Our primary contributions include addressing these fundamental questions.

\begin{enumerate}
  \item [(1)]
      As demonstrated in \cite{modares2015h}, estimating the discount factor \(\alpha\) is crucial for ensuring the existence of a nonnegative solution to the discounted Hamilton-Jacobi-Isaac (HJI) equation. In this paper, we estimate the discount factor \(\alpha\) through a comprehensive analysis of the linearized robust control system around its equilibrium point. Specifically, we compute a bound for \(\alpha\) to ensure that the Hamiltonian matrix at the equilibrium is hyperbolic, which implies that the generalized algebraic Riccati equation (GARE) \eqref{e:cHM} for the linearized robust control system has a nonnegative definite solution under natural controllability conditions (see Conditions \(\mathbf{(C_1)-(C_2)}\) below). Furthermore, with this bound for \(\alpha\), we demonstrate the existence of a nonnegative stabilizing solution to the discounted HJI equation (see Theorem \ref{t:stable} below). The bounds on \(\alpha\) obtained through this method are nearly necessary and sufficient within a semiglobal region surrounding the equilibrium, making them more precise than those presented in \cite{modares2015h}.

  \item [(2)]
      In traditional literature, it is generally accepted that feedback control approximations perform well. However, as noted in \cite{nakamura2022neural}, the `equilibrium' of the closed-loop system resulting from such approximations may become unstable or even vanish as time approaches infinity. Recently, for optimal control problems, \cite{chen2020deep} theoretically demonstrated that approximating the stable manifold is effective under certain natural conditions. However, the optimal \(H_{\infty}\) control problem is more complex because it involves estimating the \(L_2\)-gain. In this work, we show that an approximation of the stable manifold with an error \(\varepsilon > 0\) in a certain natural sense can produce a feedback control that ensures the closed-loop system maintains a finite \(L_2\)-gain of less than \(\gamma + O(\varepsilon)\). See Theorem \ref{t:finite-l2-gain} below.

  \item [(3)]
      In contrast to the HJB equation (see, e.g., \cite{van1991state, sakamoto2008, chen2020symplectic, chen2020deep}), the characteristic system of the discounted HJI equation (refer to Equation \eqref{e:HJI} below) is a contact Hamiltonian system with a state dimension of \(2n + 1\) (see Equation \eqref{e:charact-syst2} below) \cite{arnol2013mathematical} . For the specific type of contact Hamiltonian system described by \eqref{e:charact-syst2}, we can solve the first two equations and then address the third equation based on the solutions of the first two. To find trajectories for the first two equations in \eqref{e:charact-syst2}, we solve the problem near the equilibrium using a two-point boundary value problem (BVP). This local solution is then extended by solving an initial value problem (IVP). While this procedure is similar to the approach in \cite{chen2020deep}, we need more analysis to obtain the existence of the two-point BVP through a coordinate transformation inspired by \cite{sakamoto2008}.

\end{enumerate}

Additionally, based on the obtained theoretical results, we propose an adaptive deep learning algorithm that relies on randomly generated data and can be applied to high-dimensional problems, as discussed in \cite{chen2020deep}. Specifically, the training data are generated by numerically solving both the two-point BVP and the IVP mentioned earlier (see Section \ref{s:algorithm} below). We emphasize that, in addition to controlling the numerical error of the two-point BVP and IVP, error detection can be enhanced by checking the contact Hamiltonian, as suggested in \cite{sakamoto2008}, given that the value of the contact Hamiltonian remains constant along the trajectories described by Equation \eqref{e:charact-syst2}. A deep neural network (NN) is trained on this data set, and adaptive sampling is applied to add more samples near points with higher error from the previous training round.

To illustrate the effectiveness of our approach, we apply it to the optimal $H_{\infty}$ control of the Allen-Cahn parabolic partial differential equation, which is of independent interest.

\subsection{Outline}
The paper is organized as follows:
In Section \ref{s:HJI}, we revisit some preliminaries related to the discounted HJI equation for optimal $H_{\infty}$ control.
Section \ref{s:stabilizing} is dedicated to proving some fundamental properties of the stabilizing solution and the stable manifold method.
In Section \ref{s:existence}, we prove the existence of a semiglobal nonnegative stabilizing solution for the discounted HJI equation.
Section \ref{s:approx} proves the finiteness of the $L_2$-gain based on the approximate optimal controller derived from the approximation of the stable manifold.
In Section \ref{s:algorithm}, we present a deep learning algorithm designed to find an approximate stable manifold that meets the theoretical conditions established in the previous sections.
Finally, in Section \ref{s:application}, we demonstrate the effectiveness of our method by applying it to the $H_{\infty}$ control of the Allen-Cahn equation.

\section{Discounted HJI equation for robust feedback control}\label{s:HJI}
In this section, we review basic notations for robust feedback control with discounted performance values.

Consider the following nonlinear system
\begin{eqnarray}\label{e:system}
\dot x=f(x)+g(x)u+k(x)d, \quad x\in \Omega,
\end{eqnarray}
where $\Omega\subset \mathbb R^n$ is a state domain containing origin, $u\in \mathbb R^m$ is the control input, $d\in \mathbb R^l$ denotes the external disturbance, $f:\Omega\to \mathbb R^n$, $g:\Omega\to \mathbb R^{n\times m}$, $k:\Omega\to \mathbb R^{n\times l}$ are smooth functions. We assume that $f(0)=0$.
Let the performance output be
\begin{eqnarray}\label{e:p-output}
z=\left[
    \begin{array}{c}
      x \\
      u \\
    \end{array}
  \right],
\end{eqnarray}
and let its norm be
\begin{eqnarray}\label{e:norm}
\|z\|^2=x^TQx+u^TWu,
\end{eqnarray}
where $Q\in \mathbb R^{n\times n}$ and $W\in \mathbb R^{m\times m}$ are two positive-definite matrices. Moreover, we define
$
\|d\|^2=d^TGd,
$
where $G\in \mathbb R^{l\times l} $ is a positive-definite matrix.
\begin{remark}
In general, the first output component $x$ may be chosen more general as $h(x)$ (\cite[Chapter 10]{van2000l2}). In this paper, we restrict our investigation in the form \eqref{e:p-output} for simplicity of the notations.
\end{remark}

Define
$$
L^2_{\alpha}[0,\infty):=\left\{y\in L^2[0,\infty)\,\left|\, \int_{0}^{\infty}e^{-\alpha t}|y(t)|^2dt<\infty\right.\right\},
$$
where $\alpha$ is a positive discount factor.

\begin{definition}[Finite $L_2$-gain]
Let $\gamma>0$. The nonlinear system \eqref{e:system} is said to have finite $L_2$-gain less than or equal to $\gamma$ if for all $d\in L^2_{\alpha}[0,\infty)$,
\begin{eqnarray}\label{e:finite-gain}
\int_t^\infty e^{-\alpha(s-t)}\|z(s)\|^2ds\le \gamma^2 \int_t^\infty e^{-\alpha(s-t)}\|d(s)\|^2ds,
\end{eqnarray}
with $x(t)=0$.
\end{definition}
\begin{remark}
We point out that for the output $z$ as in \eqref{e:p-output}, the definition of $L_2$-gain as in \cite{van19922} \cite[Definition 1.2.1]{van2000l2} is equivalent to our definition when the discounted factor $\alpha=0$ and the bias is zero.
\end{remark}

We define the performance function by
\begin{eqnarray}\label{e:performance}
&&J(x, u,d)\\
&=&\int_t^\infty e^{-\alpha(s-t)}(x^TQ x+u^TWu-\gamma^2d^TGd)ds,\notag
\end{eqnarray}
where $x$ satisfies system \eqref{e:system} with $x(t)=x$ if $u,d$ are given.
It is well known that the robust control problem \eqref{e:system} with performance function \eqref{e:performance} can be considered as a zero-sum game. Moreover, we have the following basic result.
\begin{prop}\label{p:saddle}
For the performance function \eqref{e:performance},  the game control problem \eqref{e:system} has unique solution $(u^*,d^*)$. Moreover, this solution $(u^*,d^*)$ is a saddle solution.
\end{prop}
\begin{proof}
Denote $V(x(t))=J(x(t), u(t), d(t))$. Let
\begin{eqnarray}\label{e:value1}
H(x, V,u,d)&:=&x^T Qx+u^TWu-\gamma^T d^TGd\\
&&\quad-\alpha V+V_x^T(f+gu+kd)=0,\notag
\end{eqnarray}
where $f:=f(x)$, $g:=g(x)$, $k:=k(x)$, and $V_x:=\frac{\partial V}{\partial x}$.
The solvability of the $H_\infty$ control problem is equivalent to solvability of the following zero-sum game \cite{bacsar2008h}:
\begin{eqnarray}\label{e:minmax}
V^*(x(t))=J(x(t),u^*,d^*)=\min_u\max_d J(x(t),u,d),
\end{eqnarray}
where $x(t)$ satisfies \eqref{e:system} with $x(0)=x$ and $u,d$. From \eqref{e:minmax} and the stationary condition (\cite[Section 10.2]{lewis2012optimal}):
\begin{eqnarray}\label{e:stationary-1}
\frac{\partial H(x, V^*,u,d)}{\partial u}=0,\quad \frac{\partial H(x, V^*,u,d)}{\partial d}=0,
\end{eqnarray}
the optimal control and the worst disturbance inputs are uniquely given by
\begin{eqnarray}\label{e:u*d*}
u^*=-\frac{1}{2}W^{-1}g^TV_x^*, \quad d^*=\frac{1}{2\gamma^2}G^{-1}k^T V^*_x.
\end{eqnarray}
Direct computations yield that
\begin{eqnarray}\label{e:nash-condition}
&&V^*(x(t))=J(x(t), u^*,d^*)\\
&&=\min_{u}\max_d J(x(t), u,d)=\max_d \min_{u}J(x(t), u,d).\notag
\end{eqnarray}
Hence the game control problem \eqref{e:system} with performance function \eqref{e:performance} satisfies the Nash equilibrium condition, and then is uniquely solvable and $(u^*,d^*)$ is a saddle solution (see e.g. \cite[Section 10.2]{lewis2012optimal}). This completes the proof.
\end{proof}

Note that from \eqref{e:value1} and \eqref{e:u*d*}, $V^*$ satisfies the following discounted HJI equation
\begin{eqnarray}\label{e:HJI}
&&\bar H(x,V, V_x)\\
&:=&V^{T}_xf+V^{T}_x\left(\frac{1}{4\gamma^2}kG^{-1}k^T-\frac{1}{4}gW^{-1}g^T\right)V_x\notag\\
&&\quad\quad\quad\quad\quad\quad\quad\quad\quad\quad\quad-\alpha V+ x^T Qx=0.\notag
\end{eqnarray}
It is well known that the solution $V^*$ of \eqref{e:HJI} may be not differentiable at some points, hence $V^*$ is considered as a viscous solution to the first order HJI equation \eqref{e:HJI}. See e.g. \cite{evans2010partial}. Moreover, if $\gamma=+\infty$, then \eqref{e:HJI} becomes a Hamilton-Jacobi-Bellman (HJB) equation which corresponds to optimal control problem \cite{van2000l2}.

\section{Stabilizing solution of the discounted HJI equation} \label{s:stabilizing}

In this section, we define stable manifold of the contact Hamiltonian system of the discounted HJI equation and the associated stabilizing solution.

The characteristic system of the discounted HJI equation \eqref{e:HJI} (see e.g. \cite[Section 3.2]{evans2010partial}) is a contact Hamiltonian system of form:
\begin{eqnarray}\label{e:charact-syst1}
\left\{\begin{array}{l}
  \dot x=\bar H_p(x,V,p)\\
 \dot p=-\bar H_{V}(x,V,p)p-\bar H_x(x,V,p),\\
  \dot V=\bar H_p(x,V,p)^Tp,
\end{array}\right.
\end{eqnarray}
where $p:=V_x$.
Specifically, for \eqref{e:HJI}, the contact Hamiltonian system \eqref{e:charact-syst1} is
\begin{eqnarray}\label{e:charact-syst2}
\left\{\begin{array}{l}
  \dot x=f(x)+\left(\frac{1}{2\gamma^2}kG^{-1}k^T-\frac{1}{2}gW^{-1}g^T\right)p\\
 \dot p=\alpha p-\frac{\partial f^T(x)}{\partial x}p\\
 \quad~~-p^T\frac{\partial}{\partial x}\left(\frac{1}{4\gamma^2}kG^{-1}k^T-\frac{1}{4}gW^{-1}g^T\right)p-2 Q x\\
 \dot V=f(x)^T p+\left[\left(\frac{1}{2\gamma^2}kG^{-1}k^T-\frac{1}{2}gW^{-1}g^T\right)p\right]^T p.
\end{array}\right.
\end{eqnarray}
\begin{remark}
Note that unlike the case without a discounted factor, the system described in equations \eqref{e:charact-syst1} or equivalently in \eqref{e:charact-syst2} is a \emph{contact} Hamiltonian system (\cite{libermann2012symplectic}). Furthermore, it is important to highlight that the first two equations in \eqref{e:charact-syst2} do not rely on $V$. Consequently, the system in \eqref{e:charact-syst2} can be completely solved in two steps: initially solving the first two equations and subsequently, using this solution $(x,p)$, solving the third equation for $V$.
\end{remark}

\begin{lemma}\label{l:HC}
The contact Hamiltonian $\bar H(x,V,p)$ is invariant along the solutions to \eqref{e:charact-syst1}.
\end{lemma}
\begin{proof}
It is a direct result by taking derivative on $\bar H(x,V,p)$ with respect to $t$ along the solution of \eqref{e:charact-syst1}. See also \cite[Section 3.2]{evans2010partial}.
\end{proof}

\begin{definition}[Stabilizing solution of HJI equation]
We say that a solution $V(x)$ of the discounted HJI equation \eqref{e:HJI} is a stabilizing solution if $V_x(0)=0$ and $0$ is an asymptotically stable equilibrium of the vector field $$f(x)+\left(\frac{1}{2\gamma^2}kG^{-1}k^T-\frac{1}{2}gW^{-1}g^T\right)V_x(x).$$
\end{definition}

The next result explains the relation between the stabilizing solution of \eqref{e:HJI} and the saddle solution of the $H_{\infty}$ control problem.

\begin{prop}
Assume that $V$ is a nonnegative $C^1$ stabilizing solution of \eqref{e:HJI}. Then $V$ is the saddle solution of \eqref{e:system} with performance function \eqref{e:performance}.
\end{prop}
\begin{proof}
Since $V\ge 0$ is a stabilizing solution of \eqref{e:HJI}, we have that $(x(t),p(t))$ satisfies the characteristic system  \eqref{e:charact-syst2} with $p(t)=V_x(x(t))$ (\cite[Section 3.2]{evans2010partial}).

If we choose the input $\bar u=-\frac{1}{2}W^{-1}g^T V_x$, and the worst disturbance $\bar d=\frac{1}{2\gamma^2}G^{-1}k^TV_x$, then, from equation \eqref{e:stationary-1}, the pair $ (\bar u,\bar d)$ satisfies the Nash equilibrium condition. Hence $(\bar u,\bar d)$ is a saddle solution. The uniqueness of the saddle solution (Proposition \ref{p:saddle}) yields the conclusion.
\end{proof}

It is known that nonnegative stabilizing solution $V$ can derive a control with finite $L_2$-gain.
\begin{theorem}(\cite[Theorem 2]{modares2015h})
Assume $V(x)\ge 0$ is a solution to the discounted HJI equation \eqref{e:HJI} with $V(0)=0$. Then $\bar u$ in \eqref{e:u*d*} makes the closed-loop system \eqref{e:system} to have $L_2$-gain less than or equal to $\gamma$.
\end{theorem}

\section{Existence of semiglobal nonnegative stabilizing solution of the HJI equation}\label{s:existence}
In Section \ref{s:stabilizing} above, we shows that nonnegative stabilizing solution implies the saddle solution of the $H_{\infty}$ control system that has finite $L_2$-gain. In this section, we prove that if $\alpha$ is smaller than certain bound, then there exists nonnegative stabilizing solution in a semiglobal domain containing th origin.

\subsection{The linearized system}
Let $I_{n}$ be the identity matrix of order $n$.
Consider the linearized system of \eqref{e:system}:
\begin{eqnarray}\label{e:linearized}
&&\dot {x}=Ax+Bu+Dd,\label{e:linearized}
\end{eqnarray}
with output and control being
\begin{eqnarray}
&&\bar y=Q^{\frac{1}{2}}x,\label{e:detect}\\
&&\bar u=Lx.\label{e:pre-control}
\end{eqnarray}
Here $A=\frac{\partial f(0)}{\partial x}$, $B=g(0)$, $D=k(0)$ and $L$ is some matrix. It is clear that $(Q^{\frac{1}{2}},A)$ is detectable since $Q$ is invertible.
Define the performance function as \eqref{e:performance}.
Then for $\alpha>0$ sufficiently small, the robust control system \eqref{e:system} with locally asymptotically stable if $(u,d)$ is given by \eqref{e:u*d*}. Moreover, \cite{modares2015h} gave an estimate for upper bound for $\alpha$. In the following, by using operator method, we find another estimation of the upper bound for $\alpha$ which is sharper than that in \cite{modares2015h}.

The linear part of the first two equation of the system \eqref{e:charact-syst2} is
\begin{eqnarray}
\left[\begin{array}{c}
         \dot{x} \\
         \dot{p}
       \end{array}\right]=\left[
                            \begin{array}{cc}
                              A & -\frac{1}{2}(BW^{-1}B^T-\frac{1}{\gamma^2}DG^{-1}D^T) \\
                              -2Q & -A^T+\alpha I_n \\
                            \end{array}
                          \right]\left[
                                   \begin{array}{c}
                                     x \\
                                     p \\
                                   \end{array}
                                 \right].\notag
\end{eqnarray}
For simplicity, we denote
\begin{eqnarray}\label{e:cHM}
H_c(\alpha,\gamma)=\left[
                            \begin{array}{cc}
                              A & -\frac{1}{2}(BW^{-1}B^T-\frac{1}{\gamma^2}DG^{-1}D^T) \\
                              -2Q & -A^T+\alpha I_n \\
                            \end{array}
                          \right].
\end{eqnarray}
Then the corresponding generalized algebraic Riccati equation (GARE) is
\begin{multline}\label{e:Riccati1}
2Q+A^TP+PA-\alpha P-\frac{1}{2}P(BR^{-1}B^T-\frac{1}{\gamma^2}DG^{-1}D^T)P\\
=0.
\end{multline}
We now investigate the existence of nonnegative definite symmetric solution $P$ of GARE. Rewrite the GARE \eqref{e:Riccati1} as the following modified algebraic Riccati equation (MARE):
\begin{multline}\label{e:Riccati2}
2Q+(A-\frac{\alpha}{2} I_n)^TP+P(A-\frac{\alpha}{2} I_n)\\
-\frac{1}{2}P(BW^{-1}B^T-\frac{1}{\gamma^2}DG^{-1}D^T)P=0.
\end{multline}
The Hamiltonian matrix associated to the GARE \eqref{e:Riccati2} is
\begin{multline}
 H(\alpha,\gamma)=\\
 \left[
\begin{array}{cc}
A-\frac{\alpha}{2} I_n & -\frac{1}{2}(BW^{-1}B^T-\frac{1}{\gamma^2}DG^{-1}D^T) \\
-2Q & -A^T+\frac{\alpha}{2} I_n \\
\end{array}
\right].
\end{multline}
Letting $\gamma=+\infty$, we define
\begin{eqnarray}\label{e:H-alpha}
H_{\alpha}:= H(\alpha,\infty)=\left[
\begin{array}{cc}
A-\frac{\alpha}{2} I_n & -\frac{1}{2}BW^{-1}B^T \\
-2Q & -A^T+\frac{\alpha}{2} I_n \\
\end{array}
\right],
\end{eqnarray}
and
\begin{eqnarray}
H_0:=H(0,\infty)=\left[
\begin{array}{cc}
A & -\frac{1}{2}BW^{-1}B^T \\
-2Q & -A^T \\
\end{array}
\right].\notag
\end{eqnarray}
Set
$
T_0= \left[
                       \begin{array}{cc}
                         -I_n & 0 \\
                         0 & I_n \\
                       \end{array}
                     \right].
$
Hence
\begin{eqnarray}\label{e:H_alpha}
H_{\alpha}=H_0+\frac{\alpha}{2} T_0.
\end{eqnarray}
A necessary and sufficient condition for the existence of the stabilizing solution of the GARE \eqref{e:Riccati2} with $\alpha=0$ and $\gamma=+\infty$ is the following (\cite{sakamoto2008}, \cite[Corollary 2.4.3]{abou2012matrix}):
\begin{enumerate}
  \item[($\bf C_1$)] $H_{0}$ is hyperbolic, i.e., there is no eigenvalue of $H_{\alpha}$ on the imaginary axis.
  \item[($\bf C_2$)] $\left(A,B\right)$ is stabilizable.
\end{enumerate}

In what follows, we propose an approach to estimate $\alpha$ to ensure the existence of stabilizing solution of \eqref{e:Riccati2} under some natural assumptions.

\begin{lemma}\label{t:alpha*}
If $H_0$ satisfies ($\bf C_1$), then $H_{\alpha}$ is hyperbolic when
\begin{eqnarray}
\alpha<\delta_0:=2{\rm dist}(\sigma_-(H_0),{\rm Im}),
\end{eqnarray}
where ${\rm Im}$ represents the imaginary axis, $\sigma_-(H_0)$ denotes the set of eigenvalues of $H_0$ with negative real parts, and ${\rm dist}(\sigma_-(H_0),{\rm Im})$ denotes the distance between $\sigma_-(H_0)$ and ${\rm Im}$.
\end{lemma}
\begin{proof}
Since the Hamiltonian matrix $H_0$ is hyperbolic, it follows that the set of eigenvalues of $H_0$ is symmetric with respect to the imaginary axis (and also the real axis), and there is no eigenvalue on the imaginary axis. Hence ${\rm dist}(\sigma_-(H_0),{\rm Im})>0$. For any $\zeta\in {\rm Im}$, the resolvent
$$
R(\zeta)=(H_0-\zeta)^{-1}
$$
is well-defined.

Next, we estimate $\alpha\ge 0$ such that $H_{\alpha}-\zeta$ is invertible for all $\zeta\in {\rm Im}$. From \eqref{e:H_alpha}, the resolvent
\begin{equation}\label{e:resolvent-alpha}
R(\zeta,\alpha)=\left(H_0-\zeta+\frac{\alpha}{2} T_0\right)^{-1}
=R(\zeta)\left(I_{2n}+\frac{\alpha}{2} T_0 R(\zeta)\right)^{-1},
\end{equation}
where $I_{2n}$ is the identity matrix of order $2n$.
Using Taylor's expansion with respect to $\alpha$ in \eqref{e:resolvent-alpha}, we have
\begin{equation}\label{e:resolvent-expan}
R(\zeta,\alpha)=R(\zeta)\sum_{i=0}^\infty\left(-\frac{\alpha}{2} T_0 R(\zeta)\right)^i\\
=R(\zeta)+\sum_{i=1}^\infty \left(\frac{\alpha}{2}\right)^i (-T_0 R(\zeta))^i.
\end{equation}
By the definition of $T_0$, its matrix norm is less than $1$. Hence
\begin{eqnarray}
\|T_0 R(\zeta)\|\le \|T_0\|\|R(\zeta)\|=\|R(\zeta)\|.\notag
\end{eqnarray}
Then we have that the expansion \eqref{e:resolvent-expan} converges for $\alpha\in [0,2\|R(\zeta)\|^{-1})$.

We now estimate $\|R(\zeta)\|^{-1}$ for $\zeta\in {\rm Im}$. It is well known that
\begin{eqnarray}
\|R(\zeta)\|=\max |\sigma(R(\zeta))|,\quad \zeta\in {\rm Im},
\end{eqnarray}
where $\max |\sigma(R(\zeta))|$ denotes the maximum of the absolute value of the spectrum of $R(\zeta)$.
By the definition of $R(\zeta)$, we have that
\begin{multline}
\max |\sigma(R(\zeta))|=\max |\sigma((H_0-\zeta I)^{-1})|\\
 =(\min |\sigma(H_0-\zeta I)|)^{-1}=({\rm dist}(\sigma(H_0),\zeta))^{-1}.\notag
\end{multline}
Here $\min |\sigma(H_0-\zeta I)|$ denotes the minimum of the absolute value of the spectrum of $H_0-\zeta I$ and ${\rm dist}(\sigma(H_0),\zeta)$ is the distance from $\zeta$ to the spectral set of $H_0$.
It follows that
\begin{equation}
\|R(\zeta)\|^{-1}={\rm dist}(\sigma(H_0),\zeta)\ge \inf_{\zeta\in {\rm Im}}{\rm dist}(\sigma(H_0),\zeta)\\
={\rm dist}(\sigma(H_0),{\rm Im}).
\end{equation}
That is, for all $\zeta\in {\rm Im}$, the expansion \eqref{e:resolvent-expan} converges for all $\alpha \in [0, 2{\rm dist}(\sigma(H_0),{\rm Im}))$.
By the symmetry of the spectrum of $H_0$ with respect to the imaginary axis, ${\rm dist}(\sigma(H_0),{\rm Im})={\rm dist}(\sigma_-(H_0),{\rm Im})$. This completes the proof.
\end{proof}

\begin{remark}
From the proof of Theorem \ref{t:alpha*}, for all $\alpha\in [0, \delta_0)$, $H_{\alpha}$ keeps hyperbolicity. Hence the number of eigenvalues (counting by multiplicity) of $H_{\alpha}$ with negative real parts is $n$. It follows that the dimension of the stable manifold at $(0,0)$ is $n$ and its tangent space is the generalized eigenspace associated to the eigenvalues with negative real parts.
\end{remark}

The algebraic Riccati equation associated to the Hamiltonian matrix $H_{\alpha}$ \eqref{e:H-alpha} is
\begin{equation}\label{e:Riccati3}
2Q+(A-\frac{\alpha}{2} I_n)^TP+P(A-\frac{\alpha}{2} I_n)\\
-\frac{1}{2}PBW^{-1}B^TP=0.
\end{equation}
That is, GARE \eqref{e:Riccati2} with $\gamma=+\infty$.
Then we have the following result.
\begin{lemma}\label{l:P-alpha-infty}
If $\alpha\in [0, \delta_0)$ and conditions $(\bf C_1)-(\bf C_2)$ hold, then the Riccati equation \eqref{e:Riccati3} has a unique stabilizing solution $P_{\alpha,\infty}> 0$.
\end{lemma}
\begin{proof}
From Lemma \ref{t:alpha*}, we have that $H_{\alpha}$ is hyperbolic for $\alpha<\delta_0$. Note that since $(A,B)$ is stabilizable, $(A-\frac{\alpha}{2}I_n, B)$ is stabilizable for all $\alpha>0$. Moreover, recalling that $Q$ is positive definite, the result is obtained from \cite[Corollary 2.4.3 and Corollary 2.3.7]{abou2012matrix}.
\end{proof}

From Lemma \ref{l:P-alpha-infty}, we have that $A-\frac{\alpha}{2}I_n+BL_{\alpha,\infty}$ is asymptotically stable, where $L_{\alpha,\infty}=-W^{-1}B^TP_{\alpha,\infty}$. In general, we assume
\begin{eqnarray}\label{e:L}
\mbox{$L$ satisfies that $A-\frac{\alpha}{2}I_n+BL$ is asymptotically stable. }
\end{eqnarray}
Let $$T_L:d\to z=\left[
                  \begin{array}{c}
                    \bar y \\
                    \bar u \\
                  \end{array}
                \right],
$$
where $\bar y, \bar u$ are defined by \eqref{e:detect} and \eqref{e:pre-control}. Define the $H_{\infty}$ norm of $T_L$ as
\begin{eqnarray}
\|T_L\|^2_{H_\infty}:=\sup_{d\ne 0}\frac{\|\bar y\|^2_{L^2}+\|\bar u\|^2_{L^2}}{\|d\|^2_{L^2}}.
\end{eqnarray}
Note that $(Q^{\frac{1}{2}},A)$ is detectable since $Q$ is positive-definite. Then
the following result holds by using \cite[Theorem 8]{van1991state}.
\begin{prop}[\cite{van1991state}]\label{c:sol-riccati}
Suppose that conditions $(\bf C_1)-(\bf C_2)$ hold and $\alpha\in[0,\delta_0)$. Then it follows that
\begin{eqnarray}
\inf_{L\mbox{ satisfying \eqref{e:L}}}\|T_L\|_{H_{\infty}}<\gamma,
\end{eqnarray}
if and only if there exists a symmetric solution $P_{\alpha,\gamma}> 0$ of \eqref{e:Riccati2} such that
\begin{equation}
\sigma\left(A-\frac{\alpha}{2}I_n-BW^{-1}B^TP_{\alpha,\gamma}+\frac{1}{\gamma^2}DG^{-1}D^TP_{\alpha,\gamma}\right)
\subset \mathbb C^-,\notag
\end{equation}
where $\mathbb C^-$ denotes the set of complex numbers with negative real parts.
Moreover, $L=-W^{-1}B^TP_{\alpha,\gamma}$ is one possible $L$ satisfying \eqref{e:L} and $\|T_L\|_{H_{\infty}}<\gamma$.
\end{prop}


\subsection{Nonnegative stabilizing solution of HJI equation}

Recall that GARE \eqref{e:Riccati1} is equivalent to MARE \eqref{e:Riccati2}, and $P_{\alpha,\gamma}=\frac{\partial^2}{\partial x^2}V(0)$ where $V$ is the stabilizing solution of HJI equation \eqref{e:HJI}. Hence to ensure that the stable manifold can be represented by the graph of the value function $V$, we assume that the matrix $H_c(\alpha,\gamma)$ defined by \eqref{e:cHM} keeps hyperbolic for some $\alpha\in [0,\bar \alpha)$, where
$$\bar\alpha=\frac{1}{2}\delta_0={\rm dist}(\sigma_-(H_0),{\rm Im}).$$
Then from Lemma \ref{t:alpha*}, we have the following result.
\begin{theorem}\label{t:stable}
Suppose $\alpha\in [0,\bar\alpha)$ and conditions $(\bf C_1)-(\bf C_2)$ hold. If the problem \eqref{e:system} has finite $L_2$-gain less than $\gamma$, then there exists a neighborhood $U$ of $0\in \mathbb R^{n}$, the stable manifold for contact Hamiltonian system \eqref{e:charact-syst2} at equilibrium $(0,0,0)\in \mathbb R^{n}\times \mathbb R^n\times \mathbb R$ is $n$-dimensional and it can be represented by the graph of some smooth function $(p(x),V(x))$ in $U$.  Accordingly, there exists a smooth stabilizing solution $V\ge 0$ of the discounted HJI equation \eqref{e:HJI} in $U$.
\end{theorem}
\begin{proof}
The proof of this theorem is inspired by \cite[Theorem 9]{van1991state}. The main difference is analyzing the discounted factor $\alpha$. From Proposition \ref{c:sol-riccati} and conditions $(\bf C_1)-(\bf C_2)$ hold, for $\alpha\in [0,\bar \alpha)$, there exists a stabilizing solution $P_{\alpha,\gamma}$ of \eqref{e:Riccati2}. That is, ${\rm span}\left[
                                          \begin{array}{c}
                                            I_{n} \\
                                            P_{\alpha,\gamma} \\
                                          \end{array}
                                        \right]
$ is the generalized stable eigenspace of $H_c(\alpha,\gamma)$ \eqref{e:cHM}. Note that $H_c(\alpha,\gamma)$ is the linearized matrix of the first two equations of the contact Hamiltonian system \eqref{e:charact-syst2} and is hyperbolic. Hence by the stable manifold theorem, there exists a smooth stable manifold of the equilibrium $(0,0)\in \mathbb R^{n}\times \mathbb R^{n}$ such that its tangent space is ${\rm span}\left[
                                          \begin{array}{c}
                                            I_{n} \\
                                            P_{\alpha,\gamma} \\
                                          \end{array}
                                        \right]$.
Then in a neighborhood $U$ of $0$, the stable manifold can be represented as the graph of some smooth function $p(x)$ in $U$ with $\frac{\partial p}{\partial x}(0)=P_{\alpha,\gamma}$. From the theory of the characteristic system of first order PDE \cite[Section 3.2]{evans2010partial}, there exists a stabilizing solution $V$ of the HJI equation near the origin that the value of $V$ is determined by the third equation of \eqref{e:charact-syst2} along the characteristic line in the stable manifold, additionally,
$$V(0)=0, \frac{\partial V}{\partial x}(x)=p(x), \frac{\partial V}{\partial x}(0)=p(0),$$
$$\frac{\partial^2 V}{\partial x^2}(0)=\frac{\partial p}{\partial x}(0)=P_{\alpha,\gamma}.$$
Then the positivity of $P_{\alpha,\gamma}$ yields that $V\ge0$ in a neighborhood $U$ of the origin with $V(0)=0$ and $V(x)>0$ in $U\setminus \{0\}$. Consequently, in $U$, there exists a stable manifold of the contact Hamiltonian system \eqref{e:charact-syst2}, which can be represented by the graph of $\{(x,p,V)\in \mathbb R^{2n+1}\,|\,(x,p(x), V(x))\}$. This completes the proof.
\end{proof}

\begin{remark}
The results above indicate that in the vicinity of the equilibrium, the stabilizing solution \( V \) is smooth. However, there may be points in \( \Omega \) where \( V \) is not differentiable. Examining the regularity of \( V \) is beyond the scope of this paper. For simplicity, we will assume that \( V \) is differentiable in the semi-global domain \( \Omega \).
\end{remark}

\begin{remark}
Drawing from the preceding analysis, we derive an approximate bound for \(\alpha\) that is both nearly sufficient and necessary for ensuring the existence of a stabilizing solution for the linearized control system. It is important to note that the bound provided by \cite[Theorem 4]{modares2015h} is solely sufficient and conservative.
\end{remark}

\section{Approximation of optimal controller}\label{s:approx}
Although we prove the existence of graph representation of the stable manifold of the contact Hamiltonian system \eqref{e:charact-syst2} above, computing its analytical formula is still infeasible. Hence, in practice, numerical approximation becomes essentially important.  In this section, we analyze the robustness of the closed-loop system from certain approximate optimal control under some natural conditions.

Assume that for any fixed $\varepsilon>0$, there is an approximation $(\tilde p(x),\tilde V(x))$ of $(p(x),V(x))$ in $\Omega\subset \mathbb R^{n}$ such that
\begin{multline}\label{e:errors}
\sup_{x\in \Omega}\left|(\tilde p (x),\tilde V(x))-(p(x),V(x))\right|<\varepsilon, \\
\left|\frac{\partial \tilde p}{\partial x}(0)-P_{\alpha,\gamma}\right|<\varepsilon,
\mbox{with }(\tilde p(0),\tilde V(0))=(0,0).
\end{multline}
Then the corresponding approximate optimal controller is given by
\begin{eqnarray}\label{e:uNN}
\tilde u(x)=-\frac{1}{2}W^{-1}g^T\tilde p(x).
\end{eqnarray}
\begin{remark}
In \eqref{e:errors}, we assume that the costate $\tilde p$ satisfies that $\tilde p(0)=0$, otherwise, it may yield a finite $L_2$-gain with nonzero bias. Such condition can be satisfied by a trick as in \eqref{e:zero-p} of the algorithm below.
\end{remark}
\begin{remark}
Such kind of approximations satisfying \eqref{e:errors} does exist. For example, deep neural networks, polynomials, Fourier expansion approximations, et al. In the algorithm in Section \ref{s:algorithm} below, we will seek deep neural networks approximations.
\end{remark}

Note that from Theorem \ref{t:stable}, $p(x)=P_{\alpha,\gamma}x+O(|x|^2)$ in $B_{\rho}(0)\subset \Omega$ with some fixed $\rho>0$. Then we have the following estimate.
\begin{lemma}\label{l:u-app}
Assume that $\tilde p(x)$ satisfies \eqref{e:errors}. Then, for $x\in\Omega$,
\begin{eqnarray}\label{e:u-app}
|\tilde u(x)-u^*(x)|\le C_{W,g,\rho}\varepsilon |x|,
\end{eqnarray}
where $C_{W,g,\rho}$ depends on $W,g,\rho$, and $u^*$ is the optimal control \eqref{e:u*d*}.
\end{lemma}
\begin{proof}
From \eqref{e:errors}, we have that for $x\in B_{\rho}(0)$,
\begin{eqnarray}\label{e:x-rho}
|\tilde p(x)-p(x)|\le C_{\rho}\left|\frac{\partial \tilde p}{\partial x}(0)-P_{\alpha,\gamma}\right||x| \le C_{\rho}\varepsilon |x|.
\end{eqnarray}
where $C_{\rho}$ is a positive constant depending on $\rho$.
On the other hand, for $x\in \Omega\setminus B_{\rho}(0)$, it holds that
\begin{eqnarray}\label{e:x-rho-omega}
|\tilde p(x)-p(x)|\le \varepsilon \le \frac{1}{\rho}\varepsilon |x|,
\end{eqnarray}
Therefore, by \eqref{e:x-rho} and \eqref{e:x-rho-omega}, we have that
\begin{eqnarray}
|\tilde p(x)-p(x)| \le \bar C_{\rho}\varepsilon |x|,\quad \mbox{for all } x\in \Omega,
\end{eqnarray}
where $\bar C_{\rho}=\max\{C_{\rho},\frac{1}{\rho}\}$.
Finally, from the definition of $\tilde u$ \eqref{e:uNN} and $u^*$ \eqref{e:u*d*}, we have the conclusion.
\end{proof}

Consider the following exact closed-loop system
\begin{eqnarray}\label{e:closed-loop-optimal}
\dot x=f(x)+g(x)u^{*}(x)+k(x)d(t),
\end{eqnarray}
and its approximate system
\begin{eqnarray}\label{e:closed-loop-nn}
\dot { x}=f(x)+g( x) \tilde u(x)+k(x)d(t).
\end{eqnarray}
According to Caratheodory's existence theorem (see e.g. \cite{filippov2013differential}), if \( f(x) \), \( g(x) \), \( \tilde{u}(x) \), \( u^*(x) \), and \( k(x) \) are Lipschitz continuous with respect to \( x \), and if \( d \in L_{\alpha}^2[0, +\infty) \), then both existence and uniqueness of solutions to initial value problem for the equations \eqref{e:closed-loop-optimal} and \eqref{e:closed-loop-nn} are guaranteed.

\begin{prop}\label{t:decay-xnn}
Suppose that conditions $\mathbf{(C_1)-(C_2)}$ hold and let $\alpha\in [0,\bar\alpha)$. Assume that $\tilde p(\cdot)$ satisfies the following conditions:
\begin{enumerate}
    \item[(a)] $|\tilde p(x)-p(x)|<\varepsilon$ for all $x\in \Omega$, where $\varepsilon>0$ is sufficiently small.
    \item[(b)] $|\tilde p(x)-P_{\alpha,\gamma}x|\le \eta|x|$ for $|x|<\rho$, where $\eta>0$ is sufficiently small.
\end{enumerate}
Here, $P_{\alpha,\gamma}$ is the stabilizing solution of the generalized algebraic Riccati equation \eqref{e:Riccati1}, and $\rho>0$ is a fixed constant. Then, for $d(t)\equiv 0$, the trajectory $\tilde x(t)$ of the closed-loop system \eqref{e:closed-loop-nn} generated by the approximation of the stable manifold is sufficiently close to the exact trajectory $x(t)$ of \eqref{e:closed-loop-optimal} with the same initial state $x_0\in \Omega$, and $\tilde x(t)$ decays exponentially as $t\to+\infty$.
\end{prop}

\begin{proof}
Since $\alpha\in[0,\bar\alpha)$, from Theorem \ref{t:stable}, $\dot x=f(x)+g(x)u^*(x)$ is stable. Here note that $d\equiv0$ in \eqref{e:closed-loop-optimal} by the assumption. Then a similar proof as in \cite[Theorem 3.2]{chen2020deep} yields the conclusion.
\end{proof}

\begin{theorem}\label{t:finite-l2-gain}
Let $\varepsilon>0$ be sufficiently small constant. Then $\tilde u$ given in \eqref{e:uNN} makes the closed-loop system \eqref{e:system} to have finite $L_2$-gain less than $\gamma+O(\varepsilon)$.
\end{theorem}
\begin{proof}
Let $\tilde x(t)$ be the solution of \eqref{e:closed-loop-nn} for some given $d\in L^2_{\alpha}[0,\infty)$. From \eqref{e:errors}, it holds that
\begin{eqnarray}
&&H(V^*, \tilde u, d)\\
&=&\tilde x^T Q \tilde x +\tilde u^TR \tilde u-\gamma^2 d^Td -\alpha V^* \notag\\
&&\quad\quad\quad\quad\quad\quad\quad+V^{*T}_x (f+g\tilde u+kd)\notag\\
&=&H(V^*,u^*,d^*)+(\tilde u-u^*)^TR(\tilde u-u^*)\notag\\
&&\quad\quad\quad\quad\quad\quad\quad-\gamma^2(d-d^*)^TG(d-d^*)\notag\\
&\le& C_{W,g,\rho}\varepsilon^2\|x\|^2-\gamma^2(d-d^*)^TG(d-d^*)\notag\\
&\le& C_{W,g,\rho}\varepsilon^2\|x\|^2,\notag
\end{eqnarray}
where $C_{W,g,\rho}$ is a positive constant depending only on $W,g,\rho$ which is given in \eqref{e:u-app}.
Then taking the derivative of $V^*(\tilde x(t))$ with respect to $t$, we obtain that
$$
\frac{d}{dt} (V^*(\tilde x(t)))=V^{*T}_x\left[f(\tilde x)+g(\tilde x)\tilde u(\tilde x)+k(\tilde x)d \right](t).
$$
Compute
\begin{eqnarray}\label{e:eV^*}
&&\frac{d}{dt} (e^{-\alpha t}V^*(\tilde x(t)))\\
&=&e^{-\alpha t}\left[-\alpha V^*(\tilde x(t))\right.\notag\\
&&\quad\quad\left.+V^{*T}_x\left[f(\tilde x)+g(\tilde x)\tilde u(\tilde x)+k(\tilde x)d \right](t)\right]\notag\\
&=&e^{-\alpha t}\left[H(V^*, \tilde u, d)\right.\notag\\
&&\quad\quad\left.-\left[\tilde x^TQ \tilde x +\tilde u^TR \tilde u-\gamma^2 d^TGd\right]\right](t)\notag\\
&\le &e^{-\alpha t}\left[C_{W,g,\rho}\varepsilon^2\|\tilde x\|^2\right.\notag\\
&&\quad\quad\left.-\left[\tilde x^T Q \tilde x +\tilde u^TR \tilde u-\gamma^2 d^TGd\right]\right](t)\notag.
\end{eqnarray}
Integrating both sides of the equation  \eqref{e:eV^*}, we have that for all $T>0$,
\begin{eqnarray}
&&e^{-\alpha T}V^*(\tilde x(T))-V^*(\tilde x(0))\\
&\le&  \int_0^T -e^{-\alpha t}\left[\tilde x^T Q \tilde x +\tilde u^TR \tilde u-\gamma^2 d^TGd\right]dt\notag\\
&&\quad\quad\quad\quad\quad\quad\quad+\int_0^Te^{-\alpha t}C_{W,g,\rho}\varepsilon^2\|\tilde x\|^2dt\notag\\
&\le &\int_0^T -e^{-\alpha t}\left[\tilde x^T (Q-C_{W,g,\rho}\varepsilon^2I_n) \tilde x\right.\notag\\
&&\quad\quad\quad\quad\quad\quad\quad \left.+\tilde u^TR \tilde u-\gamma^2 d^TGd\right] dt.\notag
\end{eqnarray}
Then since $V^*\ge 0$, it follows that
\begin{eqnarray}
&&\int_0^T e^{-\alpha t}\left[\tilde x^T (Q-C_{W,g,\rho}\varepsilon^2I_n) \tilde x +\tilde u^TR \tilde u\right]dt\notag\\
&\le& \int_0^T e^{-\alpha t}[\gamma^2 d^TGd] dt + V^*(\tilde x(0)).
\end{eqnarray}
Recalling that $Q$ is positive-definite, we have that for sufficiently small $\varepsilon>0$,
\begin{eqnarray}
&&\int_0^T e^{-\alpha t}\left[\tilde x^T Q \tilde x +\tilde u^TR \tilde u\right]dt\notag\\
&\le& \int_0^T e^{-\alpha t}[(\gamma^2+C_{W,g,\rho, Q}\varepsilon^2) d^TGd] dt + V^*(\tilde x(0)).\notag
\end{eqnarray}
That is, the approximate feedback control $\tilde u$ given in \eqref{e:uNN} makes the closed-loop system \eqref{e:system} to have finite $L_2$-gain less than $\gamma+O(\varepsilon)$. This completes the proof.
\end{proof}

\section{Algorithm}\label{s:algorithm}

The theoretical results discussed above indicate that if an approximation of the stable manifold meets the conditions outlined in \eqref{e:errors}, then the corresponding approximate optimal controller is capable of stabilizing the system and exhibits a finite \( L_2 \)-gain that is less than \( \gamma + O(\varepsilon) \). In real-world applications, numerous numerical methods exist to compute suitable approximations, including polynomials and neural networks (NNs). In this section, we focus on constructing deep NN approximations for \( (p(x), V(x)) \). Compared to polynomials, deep NNs offer more flexible parameter updates and show great promise in addressing high-dimensional problems, as demonstrated in works such as \cite{han2018solving}, \cite{sirignano2018dgm}, \cite{nakamura2019adaptive}, and other references therein.

\subsection{Construction of the deep NN approximation}

Firstly, we define an NN of certain architecture, $$(p^{NN}_o,V_o^{NN}):=(p^{NN}_o(\theta; x),V_o^{NN}(\theta,x)),$$
whose input $x$ and output $p_o^{NN}$ are $n$-dimensional, and output $V_o^{NN}$ is $1$-dimensional.

In order to find an approximate NN that satisfies conditions in \eqref{e:errors}, the architecture of the original NN should be modified as follows:
\begin{multline}\label{e:zero-p}
(p^{NN}(\theta,x),V^{NN}(\theta,x))=\\
(p^{NN}_o(\theta,x)-p^{NN}_o(\theta,0),V^{NN}_o(\theta,x)-V^{NN}_o(\theta,0)).
\end{multline}
With this modification, $(p^{NN}(\theta,0), V^{NN}(\theta,0))=(0,0)$ which is necessary to satisfy the second inequality of \eqref{e:errors}.

\begin{remark}
The architecture of NN plays a crucial role in applications involving high-dimensional state spaces. For example, the complexity of a deep NN with a binary tree structure, aimed at providing approximations with a certain accuracy, depends linearly on the dimension \( n \). This is supported by \cite[Theorem 2]{poggio2017}. In contrast, the complexity of an NN with only one hidden layer increases exponentially with respect to the dimension \( n \) to achieve the same level of accuracy, as indicated in \cite[Theorem 1]{poggio2017}. In this paper, we focus primarily on control aspects, further investigation into the impact of deep NN architecture is beyond the scope of the present work.
\end{remark}

We now define the loss function to evaluate the error of the approximation stable manifold. We aim to train an NN function $(p^{NN}(\theta,\cdot), V^{NN}(\theta,\cdot))$ to fit a given dataset $\mathcal D=\left\{(x_i,p_i,V_i)\right\}_{i=1}^{|\mathcal D|}$ on $\mathcal M$ and ensure that its derivative at $0\in \mathbb R^n$ is close to $P_{\alpha,\gamma}$. Here $|\mathcal D|$ denotes the number of samples in $\mathcal D$. As in \cite{chen2020deep}, we define the following loss function for $\nu\in [1,\infty]$:
\begin{eqnarray}\label{e:loss}
&&\mathcal L^\nu(\theta;\mathcal D):=\\
&&\sigma_1\left[\frac{1}{|\mathcal D|}\sum_{i=1}^{|\mathcal D|}\|(p_i,V_i)-(p^{NN}(\theta; x_i),V^{NN}(\theta; x_i))\|^\nu\right]\notag\\
&&+\sigma_2\max_{p_i\in \mathcal D}|(p_i,V_i)-(p^{NN}(\theta; x_i),V^{NN}(\theta; x_i))|
+\sigma_3\left\|\frac{\partial p^{NN}}{\partial x}(\theta,0)-P_{\alpha,\gamma}\right\|,\notag
\end{eqnarray}
where $|\cdot|$ denotes the standard Euclidean norm in $\mathbb R^n$, $\|\cdot\|$ is the operator norm of matrix, and $\sigma_i>0$ for $i=1,2,3$ are weight constants.

To seek NN approximation of the stable manifold satisfying estimate \eqref{e:errors} in Section \ref{s:approx}, we make the loss function consist of three terms: the first term controls the mean error with exponent $\nu$ between the NN predicted value $p^{NN}(\theta; x_i)$ and the standard value $p_i$ on dataset $\mathcal D$; the second term enforces a maximum error constraint, ensuring that the maximum difference between the predicted and observed values is small; and the third term ensures that the derivative of the NN function at $0\in \mathbb R^n$ is close to the stabilizing solution $P_{\alpha,\gamma}$ of the Riccati equation \eqref{e:Riccati1}. The weight constants $\sigma_i$ control the relative importance of these terms in the overall loss function. Moreover, it is worth noting that this type of loss function is similar to the one used in \cite{chen2020deep} but distinct from traditional ones, such as the mean square error (MSE) employed in \cite{nakamura2019adaptive}, which only utilizes the first term of the loss function in equation \eqref{e:loss} with $\nu=2$.

\subsection{Trajectories generation}
For the contact Hamiltonian system \eqref{e:charact-syst2}, Theorem \ref{t:stable} establishes the existence of a stable manifold for the equilibrium $(0,0,0)\in \mathbb{R}^n\times\mathbb{R}^n\times\mathbb{R}$. To find trajectories on the stable manifold, we first address a two-point boundary value problem (BVP) in the vicinity of the equilibrium. Subsequently, we employ an initial value problem (IVP) to extend the local trajectory, following the methodology outlined in \cite{chen2020deep}.
Notably, the first two equations of the contact system \eqref{e:charact-syst2} remain independent of $V$, while the determination of the third equation in the system \eqref{e:charact-syst2} hinges on the solution of the first two equations. Therefore, our primary focus in the subsequent discussion will be on the first two equations of \eqref{e:charact-syst2}.

We first solve a two-point BVP in a neighborhood of the origin. For sake of simplicity, set $R=-\frac{1}{2}(BW^{-1}B^T-\frac{1}{\gamma^2}DG^{-1}D^T)$. Let $S$ be the solution of the Lyapunov equation
\begin{eqnarray}
(A-RP_{\alpha,\gamma})S+S(A-RP_{\alpha,\gamma})^T=R.
\end{eqnarray}
As in \cite{sakamoto2008}, setting
$$
T=\left[
    \begin{array}{cc}
      I_{n} & S \\
      P_{\alpha,\gamma} & P_{\alpha,\gamma}S+I_n \\
    \end{array}
  \right],
$$
we obtain that
\begin{equation}
T^{-1}H_{c}(\alpha,\gamma)T=
\left[
                              \begin{array}{cc}
                                A-RP_{\alpha,\gamma} & 0 \\
                                0 & -(A-RP_{\alpha,\gamma}-\alpha I_n)^T \\
                              \end{array}
                            \right].
\end{equation}
Let $B=A-RP_{\alpha,\gamma}$, $F=(A-RP_{\alpha,\gamma}-\alpha I_n)^T$, and $(x,p)^T=T(\bar x,\bar p)^T$. Then the first two equations in system \eqref{e:charact-syst2} becomes
$$
\left\{\begin{array}{l}
  \dot {\bar x}=B\bar x+N_s(\bar x,\bar p),\\
\dot {\bar p}=-F\bar p+N_u(\bar x,\bar p),
\end{array}\right.
$$
where $N_s(\bar x,\bar p)$ and $N_u(\bar x,\bar p)$ are the nonlinear terms.

Next, we consider the following two-point BVP near the equilibrium:
\begin{eqnarray}\label{e:bvp-2}
&&\left\{\begin{array}{l}
  \dot {\bar x}=B\bar x+N_s(\bar x,\bar p)\\
\dot {\bar p}=-F\bar p+N_u(\bar x,\bar p),
\end{array}\right.~~ \mbox{with }\left\{\begin{array}{l}
                                      \bar x(0)=\bar x_0, \\
                                      \bar p(+\infty)=0.
                                    \end{array}\right.
\end{eqnarray}

\begin{prop}
Suppose $|\bar x_0|$ is sufficiently small. Then the two-point BVP \eqref{e:bvp-2} has unique solution.
\end{prop}
\begin{proof}
The proof is similar as in \cite[Theorem 5]{sakamoto2008} or \cite[Theorem III.1]{chen2020symplectic}.
\end{proof}
Based on the solution $(\tilde x(t),\tilde p(t))$ of \eqref{e:bvp-2}, we solve the following IVP for $t\in [T_-,0]$ with some $T_-<0$:
\begin{eqnarray}\label{e:ivp}
&&\left\{\begin{array}{l}
  \dot x=f(x)+\left(\frac{1}{2\gamma^2}kk^T-\frac{1}{2}gW^{-1}g^T\right)p\\
\dot p=\alpha p-\frac{\partial f^T(x)}{\partial x}p\\
\quad\quad\quad-p^T\frac{\partial}{\partial x}\left(\frac{1}{4\gamma^2}kk^T-\frac{1}{4}gW^{-1}g^T\right)p-2 Q x,
\end{array}\right.\\
&&\quad\quad\quad\quad\quad\quad\quad \quad\quad\quad\quad\quad\mbox{with }\left\{\begin{array}{l}
                                      x(0)=\tilde x(0), \\
                                      p(0)=\tilde p(0).
                                    \end{array}\right.\notag
\end{eqnarray}

From Lemma \ref{l:HC}, it holds that value of the Hamiltonian is invariant along the solution of the contact Hamiltonian system \eqref{e:charact-syst1}. Therefore, the trajectories from \eqref{e:charact-syst1} lying on the stable manifold keep the Hamiltonian $\bar H\equiv 0$.

\subsection{Algorithm for approximation of stable manifold}

Building on the methodologies proposed in \cite{nakamura2019adaptive} and \cite{chen2020deep}, we employ a deep learning algorithm that leverages adaptive data generation techniques on the stable manifold. Below, we provide a brief overview of the main steps of the algorithm. For a more detailed explanation of the intricacies, please refer to \cite[Subsections 4.B-C]{chen2020deep}.

\emph{Step 0. Transformation of the model.}  We begin by rescaling the characteristic contact system \eqref{e:charact-syst2}. The necessity of such a transformation is explained in \cite[Subsection 4.C.5]{chen2020deep}.

\emph{Step 1. First generation of trajectories and sampling.}
We find a certain number of trajectories on the stable manifold by solving two-point BVP \eqref{e:bvp-2} near the equilibrium and extending the local trajectories by IVP \eqref{e:ivp} as described in \cite[Subsection 4.C.1]{chen2020deep}. Then we pick out some samples on each trajectory to obtain a training set $
\mathcal D_1:=\{(x_i,p_i, V_i)\}_{i=1}^{N_1}
$ as in \cite[Subsection 4.C.2]{chen2020deep}.
Similarly, we also generate a set of samples, $\mathcal D^{\rm val}$, for validation of the trained NN later.

\emph{Step 2. First NN training.}
Next, we train a deep neural network with a specific architecture, denoted as $(p^{NN}(\theta,\cdot),V^{NN}(\theta,\cdot))$ with parameters $\theta$, on dataset $\mathcal D_1$. The goal is to ensure that the network approximates the values $(p_i, V_i)$ for $i=1,\dots, N_1$ after training for a certain number of epochs.
If the neural network's policy function $p^{NN}(\theta,\cdot)$ achieves a performance metric ${\mathcal L^{\nu}}(\theta,\mathcal D_1)<\varepsilon$ for a chosen $\varepsilon$, we then evaluate the network on $\mathcal D^{\rm val}$ to obtain the test error.

\emph{Step 3. Adaptive data generation.} If the test error is unsatisfactory, we proceed by generating additional samples around the data points that exhibited relatively larger errors during the previous round of training. These new samples are then incorporated into the original dataset $\mathcal D_1$, resulting in an expanded dataset $\mathcal D_2$. For further information, please refer to \cite[Subsection 4.C.3]{chen2020deep}.

\emph{Step 4. Model refinement.} Based on the neural network obtained in Step 2, we proceed to further train the NN on the augmented dataset $\mathcal D_2$, following the same procedure as described in Step 2. The training process is halted once the neural network meets the criteria of a Monte Carlo test, as detailed in \cite[Subsection 4.C.4]{chen2020deep}.

\emph{Step 5. Approximate optimal feedback control.} Using the trained $p^{NN}(\theta,\cdot)$, we can derive the approximate optimal feedback control $u^{NN}$ defined in \eqref{e:uNN}. By solving \eqref{e:closed-loop-nn} with a disturbance $d\in L^2_{\alpha}[0,\infty)$, we can compute the closed-loop trajectories starting from specific initial conditions $x(0)=x_0$.

\section{Application to $H_{\infty}$ control of nonlinear parabolic PDEs}\label{s:application}
In this section, we demonstrate the effectiveness of our approach by applying it to the optimal $H_{\infty}$ control of the parabolic Allen-Cahn equation. The equation is given by:
\begin{eqnarray}
&& \mathcal X_t(\xi,t)=\sigma \mathcal X_{\xi\xi}(\xi,t)+\mathcal X(\xi,t)\label{e:parabolic1} \\
&&\quad\quad\quad\quad\quad-\mathcal X(\xi,t)^3+\mathcal U+\mathcal D, \quad \text{in }\mathcal I\times \mathbb R^+,\notag \\
&&\mathcal X(-1,t)=0, \quad \mathcal X(1,t)=0, \quad t\in \mathbb R^+, \label{e:parabolic2} \\
&&\mathcal X(\xi,0)=\mathcal X_0, \quad \xi\in \mathcal I. \label{e:parabolic3}
\end{eqnarray}
Here, $\mathcal I=[-1,1]$, $\mathcal U$ represents the control input, $\mathcal D$ represents the disturbance, and $\sigma>0$ is a viscosity constant. The performance function is defined as:
\begin{eqnarray}\label{e:performance-a}
&&J(u, d) :=
\frac{1}{2}\int_{t}^\infty e^{-\alpha(s-t)}\left[\|\mathcal X(\cdot, s)\|^2_{L^2(\mathcal I)}+\|\mathcal U(\cdot, s)\|_{L^2(\mathcal I)}^2-\gamma^2\|\mathcal D(\cdot, s)\|_{L^2(\mathcal I)}^2\right]ds. \notag
\end{eqnarray}

The Allen-Cahn equation serves as a fundamental model for phase separation phenomena (\cite{du2020phase}). Control of Allen-Cahn equations presents a common challenge in infinite-dimensional systems (\cite{colli2015optimal, chrysafinos2023analysis}). In this scenario, we aim to approximate the Allen-Cahn equation using a high-dimensional system.
The existence and uniqueness of the problem \eqref{e:parabolic1}-\eqref{e:parabolic3} is well known (e.g. \cite[Theorem 4.2]{Karafyllis2019ISS}).

\subsection{Discretization of the spatial interval}

Let $N\ge 3$ be an integer. Set $h=\frac{2}{N}$ and
\begin{eqnarray}\label{e:control-dirichlet1}
\xi_i=-1+\frac{2i}{N},\quad i=0,1,\cdots, N.
\end{eqnarray}
Then from the Dirichlet boundary condition \eqref{e:parabolic2}, it holds that
\begin{eqnarray}\label{e:control-dirichlet2}
\mathcal X(\xi_0, t)=0,\quad \mathcal X(\xi_N,t)=0, \mbox{ for }t\ge 0,
\end{eqnarray}
and for $i=1,\cdots, N-1$, the second order derivative
\begin{eqnarray}\label{e:diff-order2}
\mathcal X_{\xi\xi}(\xi_i)\approx\frac{1}{h^2}(\mathcal X(\xi_{i+1})-2\mathcal X(\xi_i)+\mathcal X(\xi_{i-1})).
\end{eqnarray}
Let the approximate state
$$X(t)=(X_1(t),\cdots, X_{N-1}(t))=(\mathcal X(\xi_1,t),\cdots, \mathcal X(\xi_{N-1},t)).$$
Hence from \eqref{e:control-dirichlet1}, \eqref{e:control-dirichlet2} and \eqref{e:diff-order2}, we get
\begin{eqnarray}
\mathcal X_{\xi\xi}\approx AX:=
\frac{1}{h^2}\left[
     \begin{array}{ccccccc}
       -2 & 1 & 0 & 0 &\cdots & 0 & 0 \\
       1 & -2 & 1 & 0 &\cdots & 0 & 0 \\
       0 & 1 & -2 & 1 &\cdots & 0 & 0 \\
       \vdots & \vdots & \vdots & \vdots &\vdots&\vdots&\vdots \\
       0 & 0 & 0 & 0 &\cdots & 1 & -2 \\
     \end{array}
   \right]\left[
            \begin{array}{c}
              X_1 \\
              X_2 \\
              X_3 \\
              \vdots \\
              X_{N-1} \\
            \end{array}
          \right].\notag
\end{eqnarray}
Accordingly, we set $u=(u_1,u_2,\cdots,u_{N-1})^T=(\mathcal U(\xi_1,t),\cdots, \mathcal U(\xi_{N-1},t))^T$ and $d=(d_1,d_2,\cdots,d_{N-1})^T=(\mathcal D(\xi_1,t),\cdots, \mathcal D(\xi_{N-1},t))$.
Then the control problem \eqref{e:parabolic1}-\eqref{e:parabolic3} becomes a discrete control system
\begin{eqnarray}
\frac{d}{dt}X=\sigma AX+X-X^3+u+d,\quad X(0)=X_0,\notag
\end{eqnarray}
where the terms $X^3=(X_1^3,\cdots, X_{N-1}^3)$, and $X_0=(\mathcal X_0(\xi_1),\, \cdots, \,\mathcal X_0(\xi_{N-1}))$. For simplicity, we set
\begin{eqnarray}
f(X)=(\sigma A+I_{N-1})X-X^3.
\end{eqnarray}
Then the problem \eqref{e:parabolic1}-\eqref{e:parabolic3} becomes
\begin{eqnarray}\label{e:dis-ac}
  \dot X=f(X)+u+d, \,\mbox{ for } t\ge 0, \quad\mbox{ with }
  X(0)=X_0.
\end{eqnarray}
The corresponding performance function of \eqref{e:performance-a} is
\begin{eqnarray}
&&\frac{1}{2}\int_{t}^\infty e^{-\alpha(s-t)}\left[\sum_{i=1}^{N-1}X_i^2(s)h+\sum_{i=1}^{N-1}u_i^2(s)h-\gamma^2\sum_{i=1}^{N-1}d_i^2(s)h\right]ds\notag\\
&=&\int_{t}^\infty e^{-\alpha(s-t)}\left[\frac{1}{N}\|X(s)\|^2+\frac{1}{N}\|u(s)\|^2-\frac{\gamma^2}{N}\|d(s)\|^2\right]ds,\notag
\end{eqnarray}
where $\|\cdot\|$ denotes the standard Euclidean norm in $\mathbb R^{N-1}$.
Let $Q=hI_{N-1}$, $R=hI_{N-1}$, $W=hI_{N-1}$.
Then the discounted HJI equation is
\begin{eqnarray}\label{e:HJI-e}
V^{T}_Xf+\frac{1}{4h}V^{T}_X\left(\frac{1}{\gamma^2}-1\right)V_X-\alpha V+ hX^T X=0
\end{eqnarray}
and the optimal control is
\begin{eqnarray}\label{e:control-f}
u^*(X)=-\frac{1}{2h}V_X.
\end{eqnarray}
Therefore the associated contact Hamiltonian system is
\begin{eqnarray}\label{e:c-ham-syst}
\left\{\begin{array}{l}
  \dot X=f(X)+\frac{1}{2h}\left(\frac{1}{\gamma^2}-1\right)P,\\
 \dot P=\alpha P-\frac{\partial f^T(X)}{\partial X}P-2 h X,\\
 \dot V=f(X)^TP+\frac{1}{2h}\left(\frac{1}{\gamma^2}-1\right)P^TP.
\end{array}\right.
\end{eqnarray}
Here the term
\begin{equation}
\frac{\partial f^T(X)}{\partial X}P=\frac{\partial }{\partial X}\left[(\sigma A+I_{N-1})X-X^3\right]^TP\\
=(\sigma A^T+I_{N-1})P-3X^2\circ P.\notag
\end{equation}
The linearized matrix of the first two equations of \eqref{e:c-ham-syst} at $(0,0)\in\mathbb R^{2n}$ is
\begin{eqnarray}
H(\alpha,\gamma)=\left[
    \begin{array}{cc}
      \sigma A+ I_{N-1} & \frac{1}{2h}\left(\frac{1}{\gamma^2}-1\right)I_{N-1} \\
      -2hI_{N-1} & -\sigma A^T-(1-\alpha)I_{N-1} \\
    \end{array}
  \right].
\end{eqnarray}
From Theorem \ref{t:stable}, for
$
0\le \alpha<\bar\alpha={\rm dist}(\sigma_-(H_0),{\rm Im}),
$
the GARE \eqref{e:Riccati1} has positive definite symmetric solution, and the HJI equation \eqref{e:HJI-e} has nonnegative solution $V$. Then the optimal control follows from \eqref{e:control-f}.

\subsection{Numerical experiment}
We are now ready to give a numerical experiment to demonstrate the effectiveness of our approach.

To begin, we construct a deep neural network based on a neural network architecture introduced by \cite[Section 4.2]{sirignano2018dgm}, which bears resemblance to a Long Short-Term Memory (LSTM) network (as in \cite{hochreiter1997long}). The deep neural network, denoted as $P^{NN}(\theta, \cdot)$, consists of 3 hidden layers, with each sub-layer comprising 60 units.

We define the activation function $\chi: \mathbb R^{60} \to \mathbb R^{60}$ as $\chi(z) = (\sin(z_1), \cdots, \sin(z_{60}))$. By employing the $\sin(\cdot)$ function as the activation function, the neural network exhibits similarities to the Fourier series. Additionally, the derivative of $\sin(\cdot)$, namely $\cos(\cdot)$, possesses broader global support compared to derivatives of standard activation functions like the sigmoid function.

In the chosen loss function \eqref{e:loss}, we set $\sigma_1 = 1$ and $\sigma_2 = \sigma_3 = 0.01$. For optimization purposes, we utilize the Adam optimizer from PyTorch.

\subsubsection{Details of the computation procedure}
We set the viscosity constant in the system \eqref{e:parabolic1} to be $\sigma=0.1$, and in the discretization, we choose $N=31$.
In the upcoming simulations, we primarily focus on two scenarios:
- Case I: $\gamma=1.2$, $\alpha=0.5\bar \alpha$.
- Case II: $\gamma=+\infty$, $\alpha=0.5\bar \alpha$.
It's worth noting that in Case II, when $\gamma=+\infty$, the original optimal $H_{\infty}$ control problem transforms into an optimal control problem. Consequently, the HJI equation \eqref{e:HJI-e} becomes a Hamilton-Jacobi-Bellman (HJB) equation.

In both cases, we generate 1500 trajectories on the stable manifold. Specifically, we select 1500 points on the sphere $\partial B_{0.8}(0)$ uniformly at random and solve the two-point BVP \eqref{e:bvp-2}. Subsequently, we extend these local trajectories by solving the IVP \eqref{e:ivp} with $T_-=0.015$. From each trajectory, we randomly pick 26 samples, including 22 samples on trajectories with positive time and 4 samples on trajectories with negative time. This process yields a training dataset comprising 39000 samples.

Additionally, we generate 500 more trajectories, resulting in 13000 samples on the stable manifold, using the same method employed for generating the testing dataset.

The neural network undergoes training on the training dataset for 4000 epochs, with a learning rate update strategy of ${\rm lr}=10^{-3}\times \left(\frac{1}{2}\right)^{[j/1500]}$, where $j$ denotes the epoch and $[j/1500]$ represents the largest integer less than or equal to $j/1500$. The training of the deep neural network is carried out on a standard laptop (Thinkpad T480s) without GPU acceleration, requiring approximately 75 minutes to complete the training procedure.

Here are the additional details of the numerical procedures for the two cases:

Case I: $\gamma=1.2$, $\alpha=0.5\bar \alpha$:
The distance from the spectrum of the Hessian matrix $H_0$ to the imaginary axis is approximately $0.553$. Therefore, we have $\bar\alpha\approx 0.553$. The distance from the spectrum of the Hessian matrix $H_{\alpha}$ to the imaginary axis is around 0.277.
By choosing $T>0$ such that $\exp(-0.277 T)\le 10^{-5}$, we find that $T_{\min}\approx 41.0$. For our simulations, we set $T_{\infty}=41$. Following the training of the neural network on the dataset, we achieve a training error of $1.7\times 10^{-3}$ and a testing error of $3.5\times 10^{-3}$.

Case II: $\gamma=+\infty$, $\alpha=0.5\bar \alpha$:
In this scenario, $\bar\alpha\approx 1.00$.  The distance from the spectrum of the Hessian matrix $H_{\alpha}$ to the imaginary axis is about $0.500$, then $T_{\min}\approx23.3$ with the same error tolerance as Case I. Setting $T_{\infty}=25$, we achieve a training error of $1.0\times 10^{-3}$ and a testing error of $1.6\times 10^{-3}$.

After the training process, we then utilize the trained neural networks to generate feedback controls using \eqref{e:uNN} in various scenarios.

\subsubsection{Simulations based on the trained NN}
We give some implementations of the trained NN. Firstly, we apply the trained NNs to control system \eqref{e:dis-ac} with disturbance $d(x,t)=(\sin t,\cdots, \sin t)$ at randomly chosen initial state. The trajectories show that the robust control (the case $\gamma=1.2$) is much better than the traditional optimal control (the case $\gamma=+\infty$).
\begin{figure}[t]
\vspace{-0.3cm}
\begin{center}
\subfigure{
\includegraphics[width=0.48\textwidth]{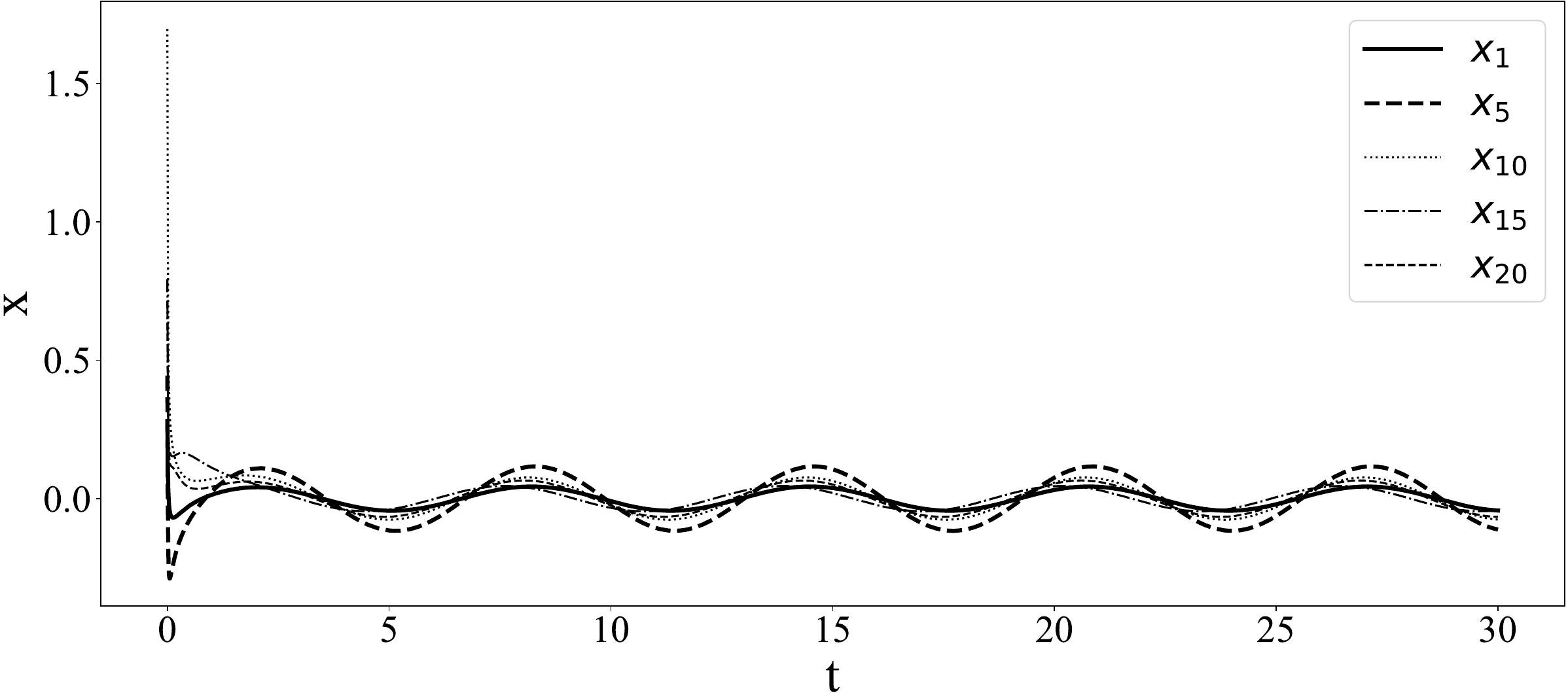}}
\subfigure{
\includegraphics[width=0.48\textwidth]{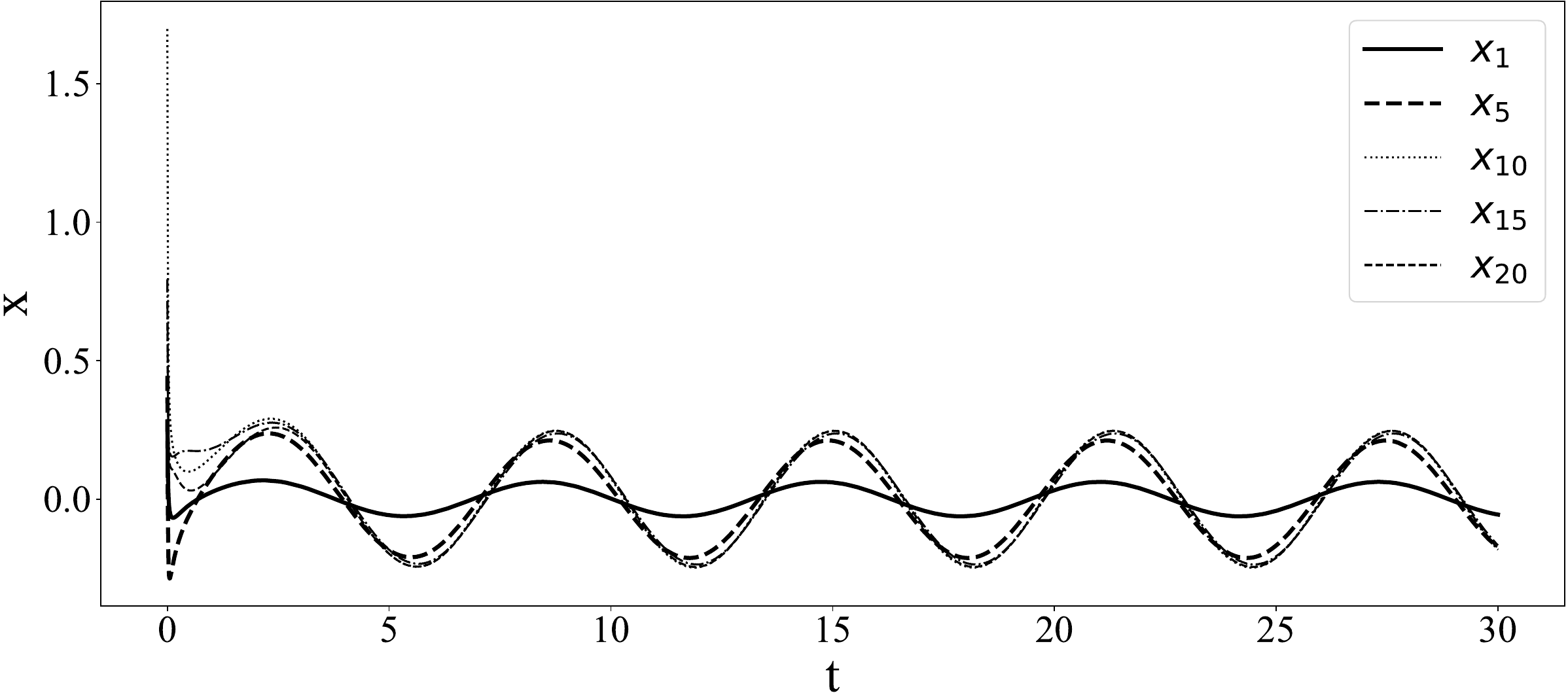}}
\end{center}
\hfill
\vspace{-0.5cm}
\caption{ The dynamics of the NN controlled system with $d(x,t)=0.3\sin t$. The first one is the trajectory from $\gamma=1.2$, and the second one is from $\gamma=+\infty$. }
 \label{f:x1}
 \vspace{-0.1cm}
\end{figure}
Furthermore,  the corresponding $|u|$ and $|w|$ are given as in Figure \ref{f:|u||w|}.
\begin{figure}[htbp]
\vspace{-0.3cm}
\begin{center}
\subfigure{
\includegraphics[width=0.48\textwidth]{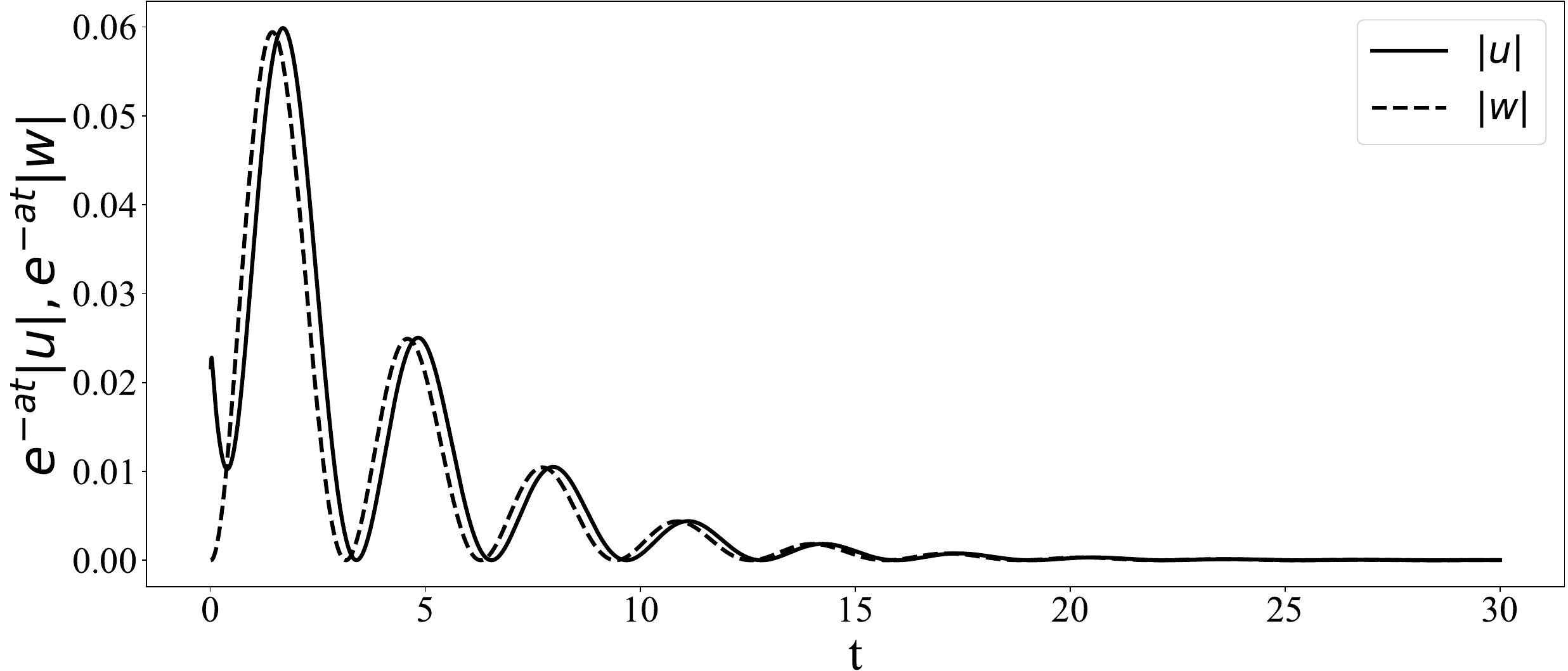}}
\subfigure{
\includegraphics[width=0.48\textwidth]{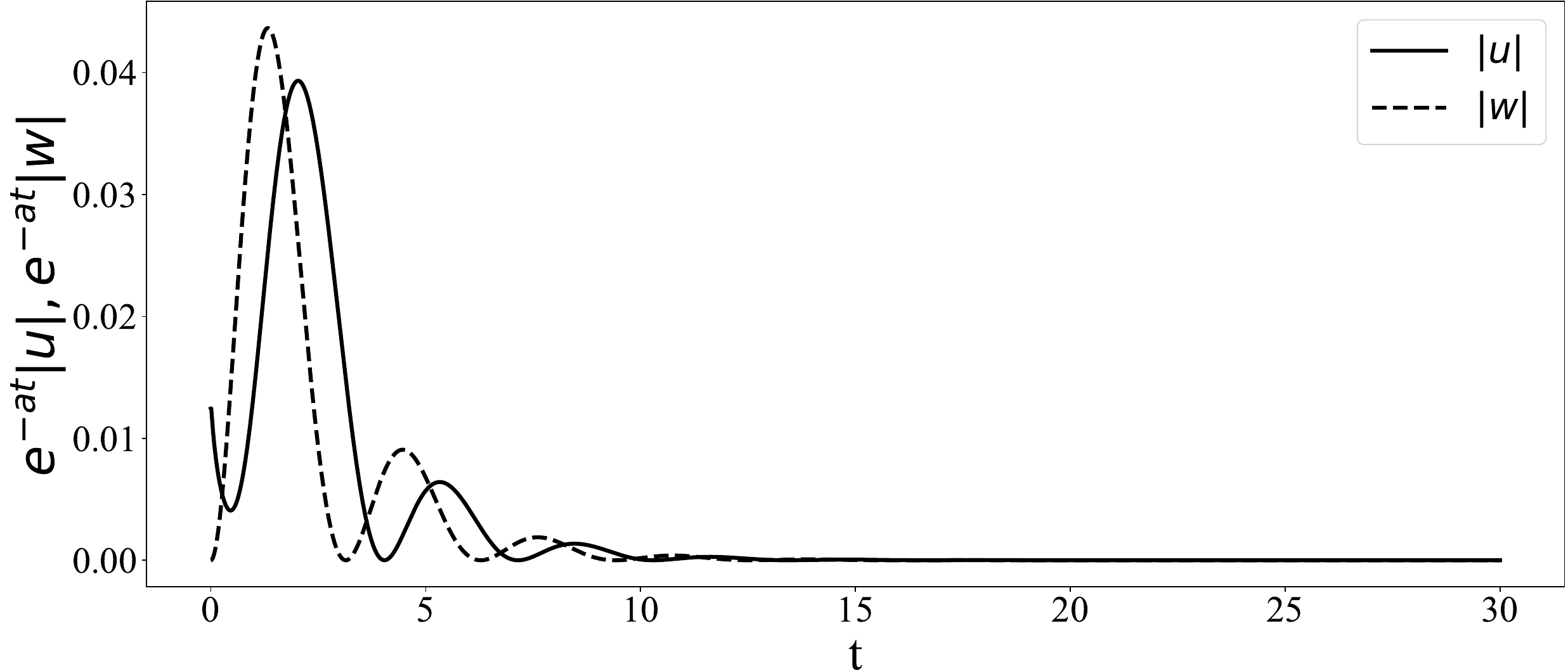}}
\end{center}
\hfill
\vspace{-0.5cm}
\caption{ The norm of $u$ and $w$ with $w(x,t)=0.3\sin t$. The first one (resp. second one) gives the norms or the case $\gamma=1.2$ (resp. $\gamma=+\infty$). }
 \label{f:|u||w|}
 \vspace{-0.1cm}
\end{figure}

For the second numerical application of the NN-generated feedback controller, we consider a tracking problem with parameters $\gamma=1.2$ and $\alpha=0.5\bar\alpha$ as a practical example. Let $r(x,t)$, for $(x,t)\in [-1,1]\times[0,30]$, represent any reference trajectory. Discretizing $x$ as previously mentioned with $N=31$, we update the states and reference states $r(x,t)$ at a frequency of 500 Hz. To elaborate further, let $s_0=1/500$ and denote $t_k = ks_0$ for $k=0,1,2,\cdots$. We define the difference $w_k(x,t) = r(x,t) - r(x,ks_0)$ as a disturbance signal, where $t\in [ks_0,(k+1)s_0)$, and $k=0,1,2,3,\cdots$.

Let $X_0$ represent the initial state, and $r_0=r(x,0)$. For each $k=0,1,2,\cdots$, we consider the following robust control systems:
\begin{multline}
\dot Y = f(Y)+u + w_k, \quad \text{for } t\in [ks_0,(k+1)s_0), \\
\text{ with }Y(t_k) = X_k - r(x,t_k).\notag
\end{multline}

In this setup, we adjust the relative state $Y$, which signifies the distance between the state and the reference states, at the update time $t_k$. Utilizing the above scheme, the NN-generated control ensures that the states $X$ effectively track the reference signal $r(x,t)$, depicted in Figure \ref{f:tracking}.

\begin{figure}[t] 
\vspace{-0.3cm}
\begin{center}
\subfigure{
\includegraphics[width=0.48\textwidth]{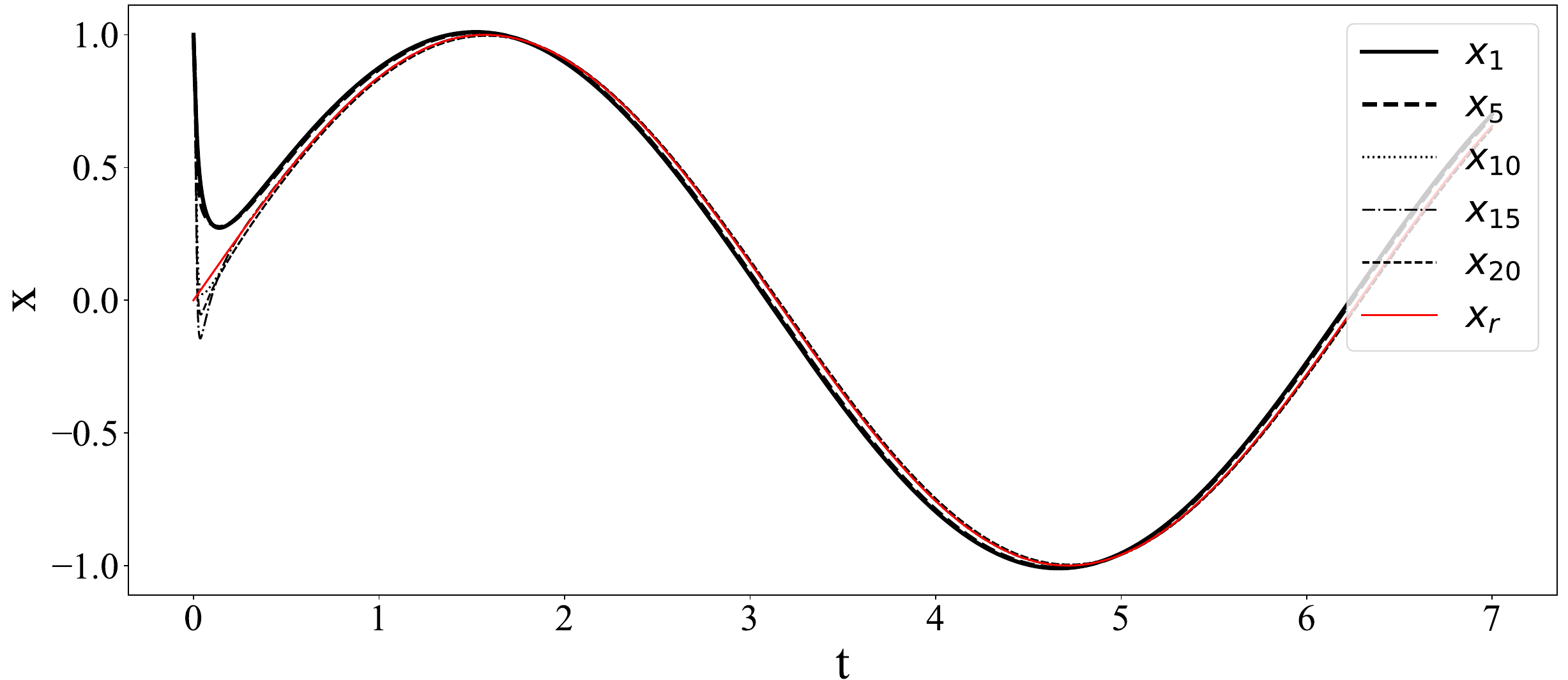}}
\end{center}
\hfill
\vspace{-0.5cm}
\caption{Tracking trajectory of $r(x,t)=\sin t$. }
 \label{f:tracking}
 \vspace{-0.1cm}
\end{figure}

\section{Conclusion}
In this paper, the discounted Hamilton-Jacobi-Isaacs (HJI) equation, which is derived from the optimal $H_{\infty}$ control on infinite time horizon with a performance function including discount term $e^{-\alpha t}$ ($\alpha>0$), has been investigated. Several fundamental issues, including theoretical analysis and appropriate algorithm, have been discussed. Specifically, a sharp estimate for the discount factor $\alpha$ has been given to ensure the existence of nonnegative stabilizing solution to the discounted HJI equation as well as the stable manifold of the contact Hamiltonian system of the discounted HJI equation. In a semiglobal domain containing the equilibrium, the stable manifold can be represented by the graph of  gradient of the stabilizing solution of the discounted HJI equation. This yields the optimal feedback control directly in this domain. Then we prove that for approximation of the stable manifold sufficiently close to the exact one in natural sense, the closed-loop system has a finite $L_2$-gain which is sufficiently close to the original one. Based on the theoretical results, we propose an algorithm to seek deep NN approximation of the stable manifold of the contact Hamiltonian system of the discounted HJI equation in a semiglobal domain containing the equilibrium. Finally, an application to optimal $H_{\infty}$ control of the parabolic Allen-Cahn equation is given to show the effectiveness of the method.

Some related problems may warrant further investigation. For instance, it would be interesting to develop a method for the general Hamilton-Jacobi equation of the form \( H(x, V, \nabla V) = 0 \), whose characteristic system is a contact Hamiltonian system. The analytical techniques used to assess the effectiveness of approximate optimal control in the present paper could potentially be generalized to other types of control problems, such as stochastic control problems. Regarding the algorithm, since we are considering the stable manifold of a contact Hamiltonian system, it is natural to develop an algorithm based on the contact structure, as suggested in \cite{chen2020symplectic}, which utilizes the symplectic structure. Moreover, our approach can be applied to high-dimensional nonlinear control problems in practice.



\end{document}